\documentclass[12pt]{amsart}
\usepackage{geometry}                
\usepackage[parfill]{parskip}    

\usepackage{graphicx}
\usepackage{amssymb,amsmath,latexsym}
\usepackage{tikz}
\usetikzlibrary{arrows,decorations.markings,decorations.pathreplacing,calc,matrix,intersections}
\tikzset{->-/.style={decoration={markings,mark=at position #1 with {\arrow{>}}},postaction={decorate}}}
\usepackage[mathscr]{euscript}
\usepackage{frcursive}

\usepackage{epstopdf}
\usepackage{ifpdf}
\usepackage{color}
\usepackage{float}

\usepackage{dsfont}
\usepackage{mathrsfs}
\definecolor{red}{rgb}{1,0,0} 

\definecolor{darkgreen}{rgb}{0, .7, 0}
\definecolor{darkblue}{rgb}{0, 0, 0.7}
\definecolor{darkred}{rgb}{0.7, 0, 0}

\definecolor{purple}{rgb}{.7, 0, 1}

\newcommand{\Z}{{\mathbb{Z}}}

\newcommand{\R}{{\mathbb{R}}}

\newcommand{\iso}{\cong}

\newcommand{\bdry}{\partial}
\newcommand{\GL}{{\mathrm{GL}}}  
\newcommand{\SL}{{\mathrm{SL}}}  
\newcommand{\Out}{{\mathrm{Out}}}  
\newcommand{\Mod}{{\mathrm{Mod}}} 

\newcommand{\CVn}{CV_n}
\newcommand{\CVns} {CV_{n,s}}
\newcommand{\BVn}{\mathcal J_n} 
\newcommand{\BVns}{\mathcal J_{n,s}}
\newcommand{\SC}{\mathscr S} 
\newcommand{\ssS}{\mathcal S}
\newcommand{\ssT}{\mathcal T}
\newcommand{\sph}{\mathfrak s}
\newcommand{\fa}{f}          
\newcommand{\sslash}{/\mkern-6mu/}
\newcommand{\VCD}{{\normalfont\scshape\lowercase{VCD}}}

\newtheorem{proposition}{Proposition}[section]
\newtheorem{definition}[proposition]{Definition}
\newtheorem{theorem}[proposition]{Theorem}
\newtheorem*{theorem*}{Theorem}
\newtheorem{lemma}[proposition]{Lemma}

 \newtheorem*{notation}{Notation}

\newtheorem{example}[proposition]{Example}

\theoremstyle{remark}
\newtheorem{remark}[proposition]{Remark}
\newtheoremstyle{red}{3pt}{3pt}{\color{red}}{}{\itshape}{.}{.5em}{}
\theoremstyle{red}

\title{The boundary of bordified Outer space}
\author{Karen Vogtmann}

\begin{document}

\begin{abstract}  Bux, Smillie and the author defined a  deformation retract $ \mathcal J_n$ of Outer space which is invariant under the action of $\Out(F_n)$ and has compact quotient; they then showed that  $ \mathcal J_n$ is homeomorphic to the Bestvina-Feighn bordification of Outer space.
We prove that the boundary of $\mathcal J_n$ has a covering by contractible pieces with contractible intersections, and we identify the nerve of this cover.   This strengthens the analogy between the Bestvina-Feighn bordification and both the classical Borel-Serre bordification of   symmetric space   and the Harvey bordification of Teichm\"uller space.  
\end{abstract}

\maketitle

\section{Introduction}
 In \cite{CV} M. Culler and the author introduced a space on which the group $\Out(F_n)$ acts as symmetries, which quickly became  known as Outer space.    Outer space was modeled both on the  Teichm\"uller space of a closed surface $S_g$, with its action of the mapping class group  $\Mod(S_g)$  and on the symmetric space of positive definite quadratic forms,  with its action of the special linear group $\SL_n(\Z)$. In both of these more classical settings the space is contractible and the action has finite point stabilizers, allowing one to conclude that the rational cohomology of the quotient by the action is an invariant of the  group.  The paper \cite{CV} proves that Outer space also has these features, i.e. it is   contractible and the action of $\Out(F_n)$ is proper.

The groups $\Mod(S_g),$ $\SL_n(\Z)$ and $\Out(F_n)$ all contain  torsion-free subgroups of finite index.  Since stabilizers of the actions on Teichm\"uller space, symmetric space  and Outer space are finite, these subgroups act freely and the integral cohomology of the quotient  is an invariant of the subgroup.  This obviously vanishes above the dimension of the space, but it turns out that it has to vanish  in   smaller dimensions as well.  The {\em cohomological dimension} of a group captures this; it is defined as the largest dimension in which the group has non-vanishing  cohomology with some coefficients.  It is not hard to show that all finite-index torsion-free subgroups of a group have the same cohomological dimension, which is called the {\em virtual cohomological dimension}  (\VCD) of the group.  

In order to determine the virtual cohomological dimension of $\SL_n(\Z)$ (and many other arithmetic groups) Borel and Serre  proved that $\SL_n(\Z)$ is a {\em virtual duality group}.  This means that there is an $\SL_n(\Z)$-module $D$ and an integer $d$ such that, for any torsion-free finite-index subgroup $\Gamma<\SL_n(\Z)$ and any $\Gamma$-module $M$,  there is a natural isomorphism
$$H^k(\Gamma;M)\iso H_{d-k}(\Gamma;M\otimes D).$$  One immediately concludes that $d$ is equal to the cohomological dimension of $\Gamma$, so is  the \VCD\ of $\SL_n(\Z)$.
By work of Bieri and Eckmann \cite{BE}, the group $H^*(\SL_n(\Z);\Z[\GL_n(\Z)])$ serves as a dualizing module if it vanishes in all dimensions except  $d$, where it is free abelian. This  group is easily identified with the cohomology with compact supports of any contractible cell complex with a proper cocompact $\SL_n(\Z)$ action (see \cite{Brown}, Proposition 11.3).   Although the action of $\SL_n(\Z)$ on the symmetric space $X_n$ is proper, it is not cocompact.  Borel and Serre  defined a partial compactification  $\widehat X_n$ of $X_n$ called the {\em Borel-Serre bordification}.  They showed that the  bordification    is a contractible manifold with boundary and  the action on $X_n$ extends to a proper cocompact action on $\widehat X_n$. They then used Poincar\'e-Lefschetz duality  to identify the cohomology with compact supports of $\widehat X_n$ with the reduced homology of the boundary $\bdry\widehat X_n=\widehat X_n\setminus X_n$, and showed that this boundary is homotopy equivalent to the associated spherical Tits building, which has the homotopy type of a bouquet of spheres.  

The action of the mapping class group $Mod(S_g)$ on Teichm\"uller space is also proper but not cocompact, and Harer was able to use the same strategy for $Mod(S_g)$ as Borel and Serre had used for $\SL_n(\Z)$, using a bordification   of Teichm\"uller space defined by Harvey \cite{Harvey} and using Poincar\'e-Lefschetz duality to identify its cohomology with compact supports  with the reduced homology of its boundary.  In this case the boundary is  homotopy equivalent to the curve complex of the surface, which Harer proved has the homotopy type of a bouquet of spheres \cite{Harer}.  

For $\Out(F_n)$ the tables were turned:  the   \VCD\ of $\Out(F_n)$ was known from \cite{CV} to be equal to $2n-3$ but it was not known whether $\Out(F_n)$ was a virtual duality group, which is a much stronger property.    Like the actions of $\SL_n(\Z)$ on its symmetric space and the action of $Mod(S_g)$ on Teichm\"uller space, the action of $\Out(F_n)$  is  proper but not cocompact. In \cite{BeFe} Bestvina and Feighn defined a partial compactification   $\widehat{CV}_n$ of Outer space $CV_n,$  analogous to the Borel-Serre bordification,  on which $Out(F_n)$ acts both properly and cocompactly.  The problem thus became to determine the cohomology with compact supports of $\widehat{CV}_n$.   Unlike symmetric spaces and Teichm\"uller space $CV_n$ is not a manifold, so Poincar\'e-Leftschetz duality is not available, and instead of studying the boundary they  found a different argument that took advantage of the fact that the \VCD\ was already known. What they proved is that $\widehat{CV}_n$ is $(2n-5)$-connected at infinity; they then applied results of Geogheghan and Mihalik   to conclude   that  $H^*_c(\widehat{CV}_n)$  vanishes in all dimensions other than $2n-3$, where it is free abelian.

In \cite{Gra} Grayson constructed an invariant deformation retract  $X_n^\varepsilon$ of the symmetric space $X_n$; this is a manifold with boundary and  the action of $\SL_n(\Z)$ on $X_n^\varepsilon$ is cocompact.  The boundary   $\bdry X_n^\varepsilon$ is    covered  by contractible sets with contractible intersections, so $\bdry X_n^\varepsilon$ is homotopy equivalent to the nerve of this cover.  This   nerve  is precisely the Tits building, so $X_n^\varepsilon$ serves the same purpose as the Borel-Serre bordification in the proof that $\SL_n(\Z)$ is a virtual duality group (it is widely assumed that $X_n^\varepsilon$ is in fact homeomorphic to the Borel-Serre bordification, but I could not locate a published proof).  The retract $X_n^\varepsilon$ avoids many of the technical difficulties that Borel and Serre encountered in adding points to the space $X_n$ and extending the action, and is generally   easier to  work with than the Borel-Serre construction.  

Inspired by Grayson's work, Bux, Smillie and the author  constructed an invariant deformation retract $\mathcal J_n$ of $CV_n$ in  \cite{BSV}, showed that it is homeomorphic to the Bestvina-Feighn bordification $\widehat{CV}_n$ of Outer space and used it to give a simpler proof that $\widehat{CV}_n$ is $(2n-5)$-connected at infinity.  

In the current paper we look more closely at   the boundary $\bdry\mathcal J_n$.  We show that $\bdry\mathcal J_n$ is covered by contractible subcomplexes  with contractible intersections (Proposition 2.10), so is homotopy equivalent to the nerve of the cover, extending the analogy with $X_n^\varepsilon$.   We  describe this nerve in terms of the combinatorial structure of $CV_n$ (Theorem 2.11), and at the end of the paper we  ask whether the analogy with symmetric spaces and Teichm\"uller spaces goes even further, i.e. whether   $\bdry\mathcal J_n$ is homotopy equivalent to a bouquet of spheres.    

The paper is structured as follows.  The space $\mathcal J_n$ decomposes into cells called {\em jewels}, and in Section ~\ref{sec:jewels}  we review the definition of a jewel and determine the structure of its facets.  In section 3 we glue   facets that are on the boundary of $\mathcal J_n$ together to form walls and show these walls are contractible with contractible intersections.  We do this using a model of Outer space introduced by Hatcher in \cite{Hat}.  Namely,  Outer space $CV_n$ decomposes naturally as a disjoint union of open simplices, and Hatcher identified these with open simplices in the {\em sphere complex} $\SC(M_n)$ of a doubled handlebody $M_n$.   We briefly review Hatcher's construction, then   reinterpret the jewels in Outer space in terms of the sphere complex and use this description to define the walls in   $\bdry\mathcal J_n$ and to show they are contractible with contractible intersections.  Finally, we identify the nerve of this covering with the subcomplex of $\SC(M_{n})$ spanned by simplices that are not in $CV_n$.

\section{Jewels associated to a core graph}\label{sec:jewels}
 
  Recall that a {\em core graph} is a graph with no isolated vertices or separating edges. It may have loops or mutiple edges, and it may be disconnected or have bivalent vertices.  By a  {\em subgraph} of a graph $G$ we mean the subgraph spanned by  a non-empty set of edges (so it cannot have isolated vertices), and a {\em core subgraph} is a subgraph which is itself a core graph.   Let $\sigma(G)$ be the regular Euclidean simplex whose vertices correspond to edges of $G$, so that subgraphs $H$ correspond to faces $\sigma(H)$ of $\sigma(G)$.  If $G$ is a core graph, we defined the {\em jewel} $J(G)$ to be the convex polytope obtained from $\sigma(G)$ by shaving faces opposite core subgraphs by uniform constants which increase with the rank of the subgraph corresponding to the face (see \cite{BSV}, Section 2). An example is shown in Figure~\ref{fig:jewel}.
 \begin{figure}
\begin{center}
\begin{tikzpicture}[scale=1] 
\fill [red!10] (105:1.1) to (60:1.9) to (35:2.5) to (25:2.5) to (-4:1.85) to (315:1.1) to (225:1.1) to (195:1.1) to (105:1.1);
\draw (90:2) to (210:2) to (330:2) to (90:2);
\draw (90:2) to (30:3) to (330:2);
\draw [dotted] (210:2) to (30:3);
\draw [red] (75:1.1) to (105:1.1) to (195:1.1) to (225:1.1) to (315:1.1) to (345:1.1) to (75:1.1);
\draw [red] (60:1.5) to (60:1.9) to (35:2.5) to (25:2.5) to (-4:1.85) to (-4:1.5) to (60:1.5);
\draw [red] (75:1.1) to (60:1.5);
\draw [red] (105:1.1) to (60:1.9);
\draw [red] (-4:1.5) to (345:1.1);
\draw [red] (-4:1.85) to (315:1.1);
\draw [red,densely dotted] (225:1.1) to (25:2.5);
\draw [red,densely dotted] (195:1.1) to (35:2.5);
\draw (0,0) .. controls (45:.75) and (135:.75)   .. (0,0);
\draw (0,0) .. controls (15:.75) and (285:.75)   .. (0,0);
\draw (0,0) .. controls (255:.75) and (165:.75)   .. (0,0);
\begin{scope}[xshift=1.6cm, yshift=.85cm, rotate=45]
\draw (0,0) .. controls (45:.75) and (135:.75)   .. (0,0);
\draw (0,0) .. controls (15:.75) and (285:.75)   .. (0,0);
\draw (0,0) .. controls (255:.75) and (165:.75)   .. (0,0);
\end{scope}
\node (a0) at (0,2.3) {$1$};
\node (b0) at (-2,-1.3) {$2$};
\node (c0) at (2,-1.3) {$4$};
\node (d0) at (2.9,1.7) {$3$};
\begin{scope}[xshift=-6cm]
\draw (.75,0) .. controls (2,1.25) and (2,-1.25)   .. (.75,0);
\draw (-.75,0) .. controls (-2,1.25) and (-2,-1.25)   .. (-.75,0);
\draw (-.75,0) .. controls (-.5,.5) and (.5,.5)   .. (.75,0);
\draw (.75,0) .. controls (.5,-.5) and (-.5,-.5)   .. (-.75,0);
\node (a) at (0,.6) {$e_2$};
\node (b) at (0,-.6) {$e_3$};
\node (c) at (2,0) {$e_4$};
\node (d) at (-2,0) {$e_1$};
\end{scope}
 \end{tikzpicture}
\end{center}
\caption{Example of a graph and its associated jewel.  The core subgraphs are $\{e_1\}$, $\{e_4\}$, $\{e_2,e_3\}$, $\{e_1,e_4\}$, $ \{e_1,e_2,e_3\}$,  and $\{e_2,e_3,e_4\}$. The faces opposite the vertices labeled $1$ and $4$ and the edge $[2,3]$ are shaved by $\varepsilon$, and the faces opposite the faces $[1,4]$, $[1,2,3]$  and $[2,3,4]$ and are shaved by $3\varepsilon$.}
\label{fig:jewel}
\end{figure}

The entire space $CV_n$ decomposes as a disjoint union of open simplices $\sigma(G,g)$, where $G$ is a  connected core graph  with no   bivalent vertices and $g$ is a homotopy equivalence from a fixed $n$-petaled rose $R_n$  to $G.$ We refer to \cite{Vog} for further background on Outer space. 

\begin{remark}
Note that core graphs do not have separating edges.  One may also include graphs $G$ with separating edges with the additional stipulation that $G$ can have no univalent vertices. The version of $CV_n$ we are using here  is  sometimes called {\em reduced Outer space}, and is the one used by Bestvina and Feighn.   Including separating edges does not change the cohomology with compact supports of the space. 
\end{remark}

\subsection{Facets of a jewel} In this section we examine the codimension one faces (the {\em facets}) of a jewel $J(G)$.  
\begin{notation} Let $H$ be a subgraph of $G$.
\begin{itemize}
\item $G\sslash H$ is the graph obtained from $G$ by collapsing each edge of $H$ to a point. 
\item   $H^\vee$ is the subgraph of $G$ spanned by the edges in $G-H$.  
\end{itemize}
\end{notation}

  If $e$ is a single edge of $G$ which is not a loop then $J(G)$ has a facet (opposite the vertex corresponding to $e$) isomorphic to $J(G\sslash e)$.  These facets are contained in the boundary of $\sigma(G)$ and are called {\em interior facets}.  All other facets  correspond to core graphs $C$ of $G,$ and are called {\em border facets}.   
We will denote the border facet opposite $\sigma(C)$  by $\fa_C$.  The following lemma records the fact that the border facets of $J(G)$ have a recursive structure.
 
 \begin{lemma}\label{facet} The border facet $\fa_C$ of $J(G)$   is isomorphic to $J(C)\times J(G\sslash C)$. 
 \end{lemma}
 
 Note that both $C$ and $G\sslash C$ are  core graphs, so $J(C)$ and $J(G\sslash C)$ are defined.
 
 \begin{proof} By Proposition 7.1 of \cite{BSV}, the codimension $(t+k)$ faces of $J(G)$ are uniquely determined by sets $S=\{e_1,\ldots,e_t, C_1,\ldots, C_k\}$, where 
\begin{enumerate}
\item the $e_i$ are edges of $G$ and $F=e_1\cup \ldots\cup e_t$ is a forest in $G$
\item $C_1\subset \ldots\subset C_k$ is a chain of proper core subgraphs, and
\item $C_i$ is the core of $F\cup C_i$.  
\end{enumerate}
The face $f_S$ determined by $S$ is a face of $f_T$ if and only if  $S\supset T$, so that $J(G)$ has an alternate description as the geometric  realization of the poset of such sets, ordered by inclusion.

Let $C$ be a core graph and  $S=\{C\}$  the  corresponding codimension one border facet, i.e. the face $f_C$   opposite  $\sigma(C)$. The isomorphism  $f_C\iso J(C)\times J(G\sslash C)$ is given by the poset map
 $$\{F,C_0\subset\ldots\subset C_{i-1}\subset C\subset C_{i+1}\subset \ldots\subset C_k\}\mapsto \hbox{\hskip 4cm}$$
$$ \hbox{\hskip 3cm} (\{F\cap C,C_0\subset\ldots\subset C_{i-1}\},\{F\sslash (F\cap C),C_{i+1}\sslash C\subset\ldots\subset C_{k}\sslash C\} ).$$
The first factor can clearly be identified with $J(C)$.  To see that the second can be identified with $J(G\sslash C)$  observe that
\begin{enumerate}
\item  the image of $F$ in $G\sslash C$ is still a forest since $C=core(F\cup C)$,   
\item    the image  $C_i\sslash C$ is still a core graph, which is 
\item  equal to the core of $F\sslash (F\cap C)\cup C_i\sslash C=(F\cup C_i)\sslash C$.   
\end{enumerate} 
  \end{proof}
  
   The facet $f_C$ is close to, and parallel to, the face $\sigma(C^\vee),$ so it is sometimes more natural to label it by  $H=C^{\vee}$; in this case we will denote it by $\fa^{H}$.   Thus$\fa^H$ is only defined when $H^\vee$ is a core graph $C$, in which case   $\fa^H = \fa_{C},$ and we say that this facet is {\em opposite} $\sigma(C)$ and {\em faces $\sigma(H)$}.

 \section{Walls in the boundary of $\mathcal J_n$}

Jewel space   $\mathcal J_n$ is the union of all marked jewels $J(g,G),$   and is homeomorophic to the Bestvina-Feighn bordification $\widehat{CV}_n$ by \cite{BSV}. The border facets form the boundary of $\mathcal J_n$.  In this section we cover this boundary by contractible subcomplexes with contractible intersections.  This is analogous to Borel and Serre's covering of the bordification of symmetric space by Euclidean spaces $e(P)$ associated to parabolic subgroups $P$.  The nerve of Borel and Serre's covering is homotopy equivalent to the associated  spherical Tits building.  In our case the nerve of the covering is homotopy equivalent to a subcomplex of the simplicial closure of Outer space, which we will identify with a certain sphere complex.  We begin by briefly reviewing the relevant definitions.    

\subsection{The sphere complex and Outer space} Outer space $CV_n$ decomposes as a union of open simplices $\sigma(G,g)$, one for each equivalence class of marked graphs $(G,g)$. Here $g\colon R_n\to G$ is a homotopy equivalence from the standard rose $R_n$, and $(g,G)$ is equivalent  to $(G',g')$ if there is a graph isomorphism $h\colon G\to G'$ with $h\circ g$ homotopic to $g'$.  The faces of some of these simplices are not in $CV_n$, namely the faces obtained by collapsing a subgraph that contains a loop, so this does not give $CV_n$ the structure of a simplicial complex.  However, adding all of the missing faces does give a simplicial complex, known as the {\em simplicial closure} $CV_n^*$ of Outer space; this is the smallest simplicial complex containing $CV_n$ as a union of open simplices.
A nice way to  understand $CV_n^*$ is by identifying open simplices in $CV_n$ with open simplices in the   {\em sphere complex} of the 3-manifold $M_n=\#_n (S^1\times S^2).$  One can then identify  the full sphere complex $\SC(M_n)$  with $CV_n^*$.   We refer the reader to \cite{Hat} for details, but sketch the ideas  below.  

An embedded $2$-sphere in a $3$-manifold  $M$ is {\em trivial} if it bounds a ball or is parallel to a boundary component. A {\em sphere system} is a set of isotopy classes of non-trivial $2$-spheres which has a set of pairwise-disjoint representatives.   The sphere complex $\SC(M)$ is the simplicial complex with  a $k$-simplex $\sigma(\ssS)$ for each  sphere system $\ssS$ in $M$ with $k+1$ elements.    

\begin{definition} A sphere system $\ssS$ is {\em complete} if and only if all components of the manifold $M-\ssS$ obtained by cutting $M$ open along $\ssS$ are simply-connected. It is {\em core} if it contains no separating spheres.  
\end{definition}

In the case $M=M_n$ a sphere system $\ssS$ is complete if  each component of $M-\ssS$ is homeomorphic to a punctured ball. Note that a complete system must have at least $n$ spheres, and a complete system with exactly $n$ spheres has exactly one complementary component.  Every maximal sphere system cuts $M_{n}$ into 3-punctured balls, so in particular is complete.   
  
\begin{proposition} [\cite{Hat}, Appendix]  Outer space $\CVn$ is the union of the interiors of simplices $\sigma(\ssS)$ in $\SC(M_{n})$ for $\ssS$ complete and core.  
\end{proposition}

The translation from the original description of Outer space as a space of marked metric graphs $(G,g)$ goes as follows. Using barycentric coordinates on the simplices $\sigma(\ssS)$,  a point in $\SC(M_n)$ can be thought of as a {\em weighted} complete sphere system, where the weight on each sphere is a positive real number, and the sum of the weights is equal to 1.   Dual to any sphere system in $M_n$ is a graph, with one vertex for each component of $M_n-\ssS$ and one edge for each sphere in $\ssS$; for example,  the dual graph to a minimal complete system is a rose and the dual graph to a maximal system is a trivalent graph.   If we fix an identification of $\pi_1(M_n)$ with $F_n$, and if all components of $M-\ssS$ are simply-connected there is an embedding of this dual graph into $M_n$ (see Figure~\ref{fig:dual}) that is unique up to homotopy and identifies the fundamental group of the graph with the fundamental group of $M_n$ up to conjugacy, i.e. it determines a marking of the dual graph.  The weight on the sphere intersecting a given edge can be interpreted as the length of the edge; thus a sphere system determines  a marked metric graph with volume one.  

\begin{figure}
\begin{center}
\begin{tikzpicture} [scale=.5]
   \draw []  plot [smooth cycle, tension=.8] coordinates 
 { (-6,0)   (-4.5 ,2)  (-2,1.25)   (0,2)   (2,1.25)   (4.5,2)  (6,0) (4.5,-2) (2,-1.25)  (0,-2) (-2,-1.25) (-4.5,-2)
    };
\draw  (.70,.1) arc (25:155:.75);
\draw   (.75,.3) arc (-10:-170:.75);
\begin{scope}[xshift=3.75cm]
\draw (.70,.1) arc (25:155:.75);
\draw (.75,.3) arc (-10:-170:.75);
\end{scope}
\begin{scope}[xshift=-3.75cm]
\draw (.70,.1) arc (25:155:.75);
\draw  (.75,.3) arc (-10:-170:.75);
\end{scope}

 \draw[thick, red, fill=red!10] (5.2,0) ellipse (.8cm and .25cm); 
 \node[right] (s4) at (6,0) {$\sph_4$};
 \draw[thick, red,  fill=red!10] (-5.2,0) ellipse (.8cm and .25cm); 
  \node[left] (s1) at (-6,0) {$\sph_1$};
   \draw[thick, red,  fill=red!10] (0,1.25) ellipse (.25cm and .7cm); 
    \node[above] (s2) at (0,2) {$\sph_2$};
 \draw[thick, red,  fill=red!10] (0,-1.15 ) ellipse (.25cm and .82cm); 
     \node[below] (s2) at (0,-2) {$\sph_3$};
 
\draw [fill=blue] (-2,0) circle(.05);
\draw [fill=blue] (2,0) circle(.05);
\draw [blue] (-2,0) .. controls (-6,3) and (-6,-3) .. (-2,0);
\draw [blue] ( 2,0) .. controls ( 6,3) and ( 6,-3) .. ( 2,0);
\draw [blue] ( -2,0) .. controls ( -1,1.5) and ( 1,1.5) .. ( 2,0);
\draw [blue] ( -2,0) .. controls ( -1,-1.5) and ( 1,-1.5) .. ( 2,0);
\end{tikzpicture} 
\end{center}
\caption{ A complete core sphere system $\mathcal S=\{\sph_1,\sph_2,\sph_3,\sph_4\}$ in $M_3=\#_3 (S^1\times S^2)$ (double the pictured handlebody and disks) and its dual graph.  }
\label{fig:dual}
\end{figure} 

 In the other direction, given a marked metric graph with fundamental group $F_n$, put a point in the middle of each edge,
fatten the graph to a handlebody (so the points become disks), double the handlebody to get a manifold homeomorphic to $M_n,$ (so the disks become $2$-spheres).  Now   use the marking to construct a homeomorphism to $M_n$.  The images of the $2$-spheres  are a sphere system in $M_n$.  

One advantage of this description of $\CVn$ is that a single sphere can be a member of many different, incompatible sphere systems, i.e. an edge of a marked graph has a life independent of any particular marked graph containing it.

Another advantage is that the proof that  $\CVn$ 
is contractible using surgery paths in sphere systems \cite{HV} works with only minor modifications for the analogous space $\CVns$ defined using $M_{n,s}=M_n\backslash \coprod_{i=1}^s B_s$, the manifold obtained by removing $s$ disjoint balls from $M_n.$  Then $\CVns$ is equal to the subspace of $\SC(M_{n,s})$ consisting of the union of the interiors of simplices $\sigma(\ssS)$ for $\ssS$ complete and core.  This can also be interpreted as a space of marked graphs with leaves; to see the idea, take the dual graph of a sphere system, then  attach a  leaf for each boundary sphere $\mathfrak s$, emanating from  the vertex corresponding to the component of $M_{n,s}-\ssS$ that contains $\mathfrak s$.

Any diffeomorphism of a 3-manifold $M$ sends disjoint $2$-spheres to disjoint $2$-spheres, so acts on $\SC(M)$. By a theorem of Laudenbach \cite{Lau}, $Out(F_n)$ is isomorphic to $\pi_0(\hbox{Diff}(M_n))/D$, where $D$ is a finite group which acts trivially on isotopy classes of $2$-spheres  ($D$ is generated by a finite number of ``Dehn twists in $2$-spheres").  Under the identification of the space of marked metric graphs with the space of weighted sphere systems, the induced action of $Out(F_n)$ on $\SC(M_n)$ becomes the old action of $Out(F_n)$ on $CV_n$.   For $s>0$, define   $\Gamma_{n,s}$ to be the group of isotopy classes of diffeomorphisms of $M_{n,s}$ fixing the boundary, modulo Dehn twists on $2$-spheres (or, equivalently, the group of homotopy classes of homotopy equivalences of a rose with $s$ leaves that fix the leaves).  Then $\Gamma_{n,s}$ acts on $\SC(M_{n,s})$, and the stabilizer of a point permutes the spheres (preserving weights), so is always finite.  In addition to the fact that $\Gamma_{n,0}\iso Out(F_n)$, we also have    $\Gamma_{n,1}\iso Aut(F_n)$   .

\begin{definition} The space $\CVns$ is the union of the interiors of simplices $\sigma(\ssS)$ in $\SC(M_{n,s})$ for $\ssS$ complete and core.  
\end{definition}

Note that $\CVns$ is  not a subcomplex of $\SC(M_{n,s})$, just a subspace, since each simplex is missing some faces (in particular, all of its vertices are missing).   The entire simplicial complex $\SC(M_{n,s})$ is   the  simplicial closure  of $CV_{n,s}$ and we will also use the notation $CV_{n,s}^*$.

\subsection{Sphere system description of $\mathcal J_n$}

We now describe the jewel retract $\BVns$ in the language of sphere systems.  Since the jewel associated to $J(G,g)$ is obtained by truncating certain faces of $\sigma(G,g)$ by amounts satisfying certain constraints, we begin by identifying the relevant faces and constants.  
\begin{definition} Let $\ssS$ be a core sphere system in $M=M_{n,s}$. A subsystem $\ssT\subset \ssS$ is a {\em core complement} if $\ssS-\ssT$ is core in $M-\ssT$.   \end{definition} 
See Figure~\ref{fig:corecomplement} for an example. 
\begin{figure}
\begin{center}
\begin{tikzpicture} [scale=.5]
   \draw []  plot [smooth cycle, tension=.8] coordinates 
 { (-6,0)   (-4.5 ,2)  (-2,1.25)   (0,2)   (2,1.25)   (4.5,2)  (6,0) (4.5,-2) (2,-1.25)  (0,-2) (-2,-1.25) (-4.5,-2)
    };
\draw  (.70,.1) arc (25:155:.75);
\draw   (.75,.3) arc (-10:-170:.75);
\begin{scope}[xshift=3.75cm]
\draw (.70,.1) arc (25:155:.75);
\draw (.75,.3) arc (-10:-170:.75);
\end{scope}
\begin{scope}[xshift=-3.75cm]
\draw (.70,.1) arc (25:155:.75);
\draw  (.75,.3) arc (-10:-170:.75);
\end{scope}

 \draw[thick, red, fill=red!10] (5.2,0) ellipse (.8cm and .25cm); 
 \node[right] (s4) at (6,0) {$\sph_4$};
 \draw[thick, red,  fill=red!10] (-5.2,0) ellipse (.8cm and .25cm); 
  \node[left] (s1) at (-6,0) {$\sph_1$};
  \draw[thick, red, fill=red!10] (-1.9,0) ellipse (1.3cm and .25cm); 
    \node (s2) at (-2,.75) {$\sph_2$};
 \draw[thick, red,  fill=red!10] (0,-1.15 ) ellipse (.25cm and .82cm); 
     \node[below] (s2) at (0,-2) {$\sph_3$};
 
\draw [fill=blue] (-2,-.75) circle(.05);
\draw [fill=blue] (2,0) circle(.05);
\draw [blue] (-2,-.75) .. controls (-10,-1) and (0,3.25) .. (2,0); 
\draw [blue] ( 2,0) .. controls ( 6,3) and ( 6,-3) .. ( 2,0); 
\draw [blue] ( -2,-.75) .. controls ( -1,1.5) and ( 1,1) .. ( 2,0); 
\draw [blue] ( -2,-.75) .. controls ( -1,-1.5) and ( 1,-1.5) .. ( 2,0); 
\end{tikzpicture} 
\end{center}
\caption{The core complements in $\mathcal S=\{\sph_1,\sph_2,\sph_3,\sph_4\}$ are $\{\sph_1\},\{\sph_2\},\{\sph_3\},
\{\sph_4\}, \{\sph_1, \sph_4\}, \{ \sph_2, \sph_4\}$ and $\{\sph_3,\sph_4\}.$}
\label{fig:corecomplement}
\end{figure} 

For any sphere system $\ssS$ in $M=M_{n,s}$, define $h(\ssS)$ to be the number of spheres in $\ssS$ minus the number of complementary components:
$$h(\ssS)=\#\ssS- \hbox{rank}(H_0(M-\ssS)).$$
(In terms of graphs, $h(\ssS)$ is ``edges minus vertices," i.e. the negative of the Euler characteristic of the graph dual to $\ssS$.)  

\begin{example}  If $\ssS=\{\sph\}$ is a single non-separating sphere, then $h(\sph)=0$. If $\ssS$ is complete, then $h(\ssS)=n-1$.    
\end{example}

If $\ssT$ is a core complement in $\ssS$, let $r(\ssS,\ssT)=h(\ssS)-h(\ssT) \hbox{ and }$
 $$t(\ssS,\ssT)= \frac{3^{r(\ssS,\ssT)}}{N}$$ for some fixed $N\gg n$.

We can now construct the jewel of a (not necessarily complete)  core sphere system:

\begin{definition} Let $\ssS$ be a core sphere system.
 The jewel $J(\ssS)\subset \sigma(\ssS)$ is the polytope obtained by truncating the  face $\sigma(\mathcal T)$ whenever $\ssT\subset \ssS$ a core complement, by the amount $t(\ssS,\ssT)$.  The codimension one face of $J(\ssS)$ obtained by truncating $\sigma(\ssT)$ is called the {\em facet} $\varphi(\ssS,\ssT)$. 
\end{definition}

\begin{lemma} Let $G$ be the graph dual to $\ssS$.  Then $J(G,g)$ is isometric to $J(\ssS)$. 
\end{lemma}
\begin{proof} The subsystem $\ssT\subset \ssS$ is a core complement if and only if the edges dual to $\ssS-\ssT$ form a core sub graph of the graph dual to $\ssS$. Furthermore, the truncation constants satisfy the conditions   used in the proof of Proposition 2.2 of \cite{BSV}, which characterizes the jewel $J(G,g)$.
 \end{proof}

If we glue together all jewels associated to complete systems in $M_{n,s}$ we obtain the jewel space $\BVns=\mathcal J(M_{n,s}),$ as explained in \cite{BSV}. We are interested in studying the  {\em boundary}  $\bdry\BVns,$ i.e. the union of the  facets $\varphi(\ssS,\ssT)$ for all complete core $\ssS$ and core complements $\ssT\subset \ssS$. 

\begin{lemma}\label{ccore}  Let $\ssT$ be an incomplete core system in $M$.  Then $\ssT$ is a core complement for a  core system $\ssS$ if and only if $\ssS-\ssT$ is a  core system in $M-\ssT$.  
\end{lemma}
\begin{proof} This follows from the definition of core complement, together with the observation that if a sphere in $\ssS-\ssT$ does not separate $M-\ssT$, then it does not separate $M$.  
\end{proof}

If $\ssT$ is any subsystem of a core system $\ssS$, then    $\ssS$ is complete in $M$ if and only if $\ssS-\ssT$ is  complete   in $M-\ssT$.  The lemma thus shows that there are lots of complete core systems $\ssS$ containing a given  incomplete system $\ssT$ as a core complement.

The boundary $\bdry\BVns$ is the union of {\em walls} $W(\ssT)$ for $\ssT$ incomplete,  where
$$W(\ssT)=\bigcup_{\ssS \hbox{ \tiny complete core, }\\ \ssS\supset \ssT}\varphi(\ssS,\ssT)$$
 
 To understand $W(\ssT)$, we translate Lemma~\ref{facet} into the language of sphere systems:
 
 \begin{lemma} \label{sphere_facet}Let $\ssT$ be a core complement to a core system $\ssS$ in $M$.  Then $$\varphi(\ssS,\ssT)\iso J_M(\ssT)\times J_{M-\ssT}(\ssS-\ssT).\qed$$ 
 \end{lemma}
 
 
 \begin{proposition}\label{product} For any incomplete core system $\ssT$ in $M$, $W(\ssT)$ is contractible.
 \end{proposition}
 
\begin{proof}  By Lemma~\ref{sphere_facet} we have $$W(\ssT)\iso  J_M(\ssT)\times \bigcup_{_{\ssS \hbox{ \tiny complete core, }\\ \ssS\supset \ssT}} J_{M-\ssT}(\ssS-\ssT)\iso J_M(\ssT)\times \mathcal J({M-\ssT}).$$ 
$J_M(\ssT)$ is a truncated simplex, so is contractible.   If $M-\ssT$ has $k$ components $M^1\cup\ldots\cup M^k$, then  $\mathcal J(M-\ssT)\iso \mathcal J(M^1)\ast\ldots\ast \mathcal J(M^k)$.  Since each $\mathcal J(M^i)$ is a deformation retract of some $CV_{n_i,s_i}$, each is contractible.  
 \end{proof}

 \begin{theorem}  $\bdry\BVns$ is homotopy equivalent to  the subcomplex $\SC_\infty(M_{n,s})$ of the complex $\SC(M_{n,s})$of core sphere systems spanned by incomplete  systems.
 \end{theorem}
 
 \begin{proof}   Note that any core complement $\ssT$ is necessarily incomplete, since any sphere in $M-\ssT$ for a complete system $\ssT$ is separating. Therefore letting $N\to\infty$ in the definition of the truncation constant $t(\ssS,\ssT),$ so that $t(\ssS,\ssT)\to 0,$ gives a map  
$$\bdry\BVns\to \SC_\infty(M_{n,s}).$$ By Proposition~\ref{product} the pre-image of an open simplex $\sigma(\ssT)$ in $\SC_\infty(M_{n,s})$ is homeomorphic to   $(J_M(\ssT))^0\times \mathcal J({M-\ssT})$, where $(J_M(\ssT))^0$ denotes the interior of the jewel $J_M(\ssT)$.  Since these pre-images are all contractible, the map is a homotopy equivalence.   
 \end{proof}
 
\begin{remark} The subcomplex $\SC_\infty(M_{n})$ contains the entire $(n-2)$-skeleton of $\SC(M_{n}),$ since it takes at least $n$ spheres to cut $M_{n}$ into   simply connected pieces.  Since $\SC(M_{n})$ is contractible, this shows that $\SC_\infty(M_{n})$ is at least $(n-3)$-connected.  The topology of $\SC_\infty(M_{n})$ and the analogous complex for unreduced Outer space was studied by B. Br\"uck and R. Gupta \cite{BrGu}, whose methods give another proof that  $\SC_\infty(M_{n})$ is $(n-3)$-connected, and also show that the unreduced version is $(n-2)$-connected.
\end{remark} 
It is easy to compute, using the Euler characteristic, that  $\SC_\infty(M_{n})$ is $(3n-6)$-dimensional.  This situation reflects that of the arc complex $\mathcal A$ of a surface with boundary.  The  complex $\mathcal A$ embeds naturally into $\SC(M_{n})$ for appropriate $n$ (take the product of the surface with an interval to get a handlebody, then double it;  the arcs become discs, then spheres).   For example, if the surface has  genus $g>0$ with one boundary component, the entire arc complex $\mathcal A$ is contractible of dimension $6g-4$,   the subcomplex $\mathcal A_\infty$ at infinity is $(6g-6)$-dimensional and $\mathcal A_\infty$ contains the $(2g-2)$-skeleton so is $(2g-3)$-connected.  Harer proves that $\mathcal A_\infty$ is in fact homotopy equivalent to a bouquet of $(2g-2)$-spheres \cite{Harer}, and we ask whether the same is true here:

{\bf Question:}  Is $\SC_\infty(M_{n})$   homotopy equivalent to a bouquet of spheres of dimension $n-2$?

Recall that $\Out(F_n)$ is a virtual duality group with virtual cohomological dimension $2n-3$, so the dualizing module is $H^{2n-3}(Out(F_n);\Z(Out(F_n))\iso H^{2n-3}_c(\BVn)$.  The space $\BVn$ has dimension $3n-4$, so if $\BVn$ were an orientable manifold  Poincar\'e duality would identify $H^{2n-3}_c(\BVn)$ with $H_{n-1}(\BVn,\bdry \BVn)$, which is isomorphic to $H_{n-2}(\bdry\BVn)\iso H_{n-2}(\SC_\infty(M_n))$.  But   $\BVn$ is  not a manifold, and we ask

{\bf Question:}  What (if any) is the relation between $H_{n-2}(\SC_\infty(M_{n}))$ and  $H^{2n-3}_c(\BVn)$?  

\begin{remark} For recent work on the structure of the dualizing module $H^{2n-3}_c(\BVn)$ for $\Out(F_n)$, see the  paper by R. Wade and T. Wasserman \cite{WaWa}.
\end{remark}

Both of these questions can also be asked for $\SC_{\infty}(M_{n,s})$ (where the spheres would have dimension $n-3+s$) and $\mathcal J_{n,s}$. 

\subsection{The Jacobian map}
There is a natural map, called the {\em Jacobian } map, from $CV_n$ to the symmetric space $Q_n=\SL(n,\mathbb R)/SO(n)$ of positive definite quadratic forms on $\R^n$ (see \cite{Bak}).   This map is  simple to define using the description of a point of $CV_n$ as a graph $X$  together with a {\em marking} $g$ which identifies $\pi_1(X)$ with $F_n$.  Specifically, equip the Euclidean space $\mathbb R^E$ spanned by the edges $E$ of $X$ with the standard inner product, then restrict this to the subspace $H_1(X;\R)\subset \R^E$ to get a positive definite quadratic form on $H_1(X;\R)$.  The marking $g$ also identifies $H_1(X;\R)$ with $\mathbb R^n$, so we can think of this form as a point of $Q_n$.   

{\bf Question:} What is the image of $\mathcal J_n$ under the Jacobian map?  What is the image of  $\bdry\BVn$?
In particular, does the Jacobian map send the Jewel space to the Grayson retract? 
 
The Jacobian map commutes with the actions of $\Out(F_n)$ on $CV_n$ and and $GL(n,\Z)$ on $Q_n$, and the induced map on the quotients is called the {\em tropical Torelli map} by algebraic geometers (see, e.g., \cite{Chanetc}).  The Jacobian map extends naturally to appropriate bordifications of $CV_n$ and $Q_n$, and the corresponding Torelli maps were studied recently in papers by F. Brown \cite{FBrown} and F. Brown, M. Chan, S. Galatius and S. Payne 
\cite{BCGP}, who used them in particular to detect new unstable cohomology classes for $GL(n,\Z)$ and $SL(n,\Z)$.

\end{document}